\documentclass[11pt,a4paper,reqno]{amsart}
\usepackage{amsfonts}
\usepackage{amsmath}
\usepackage{amssymb}
\usepackage{hyperref}
\usepackage{graphicx}
\usepackage{bbm}

\usepackage{epstopdf}
\usepackage{color}

\usepackage{amsthm}
\usepackage{enumitem}

\usepackage{mathrsfs}
\usepackage{dutchcal}

\newcommand{\cE}{\mathcal E}
\newcommand{\E}{\mathcal E}
\newcommand{\cP}{\mathcal P}

\newcommand{\tail}{\operatorname{tail}}
\newcommand{\head}{\operatorname{head}}

\newcommand{\dd}{\mathrm d}

\numberwithin{equation}{section}
\newtheorem{theorem}{Theorem}[section]
\newtheorem{lemma}[theorem]{Lemma}
\newtheorem{remark}[theorem]{Remark}
\newtheorem{corollary}[theorem]{Corollary}
\newtheorem{proposition}[theorem]{Proposition}
\newtheorem{definition}[theorem]{Definition}

\allowdisplaybreaks

\def\enddoc{\end{document}}

\begin{document}
\author{Qingsong Gu}
\address{Department of mathematics, Nanjing University, Nanjing 210093, P. R. China} \email{qingsonggu@nju.edu.cn}

\author{Lu Hao}
\address{Universit\"{a}t Bielefeld, Fakult\"{a}t f\"{u}r Mathematik, Postfach 100131, D-33501, Bielefeld, Germany}
\email{lhao@math.uni-bielefeld.de}

\author{Xueping Huang}
\address{School of Mathematics and Statistics, Nanjing University of Information Science and Technology,
	Nanjing 210044, P. R. China}
\email{hxp@nuist.edu.cn}
	
\author{Yuhua Sun}
\address{School of Mathematical Sciences and LPMC, Nankai University, 300071
		Tianjin, P. R. China}
\email{sunyuhua@nankai.edu.cn}
\title[$p$-Laplace inequality]{A Volume-Growth Criterion for the $p$-Laplace Inequality on Weighted Graphs}

\thanks{\noindent
 Q. Gu was supported by the National Natural Science Foundation of China (Grant No. 12101303 and 12171354).
 L. Hao was funded by the Deutsche Forschungsgemeinschaft (DFG, German Research Foundation), Project-ID 317210226, SFB 1283.
 X. Huang was supported by
	the National Natural Science Foundation of China (Grant No. 11601238).
 Y. Sun was funded by the National Natural Science Foundation of
	China (Grant No. 12371206) and the Fundamental Research Funds for the Central Universities, No. 050-63263078.}

\begin{abstract}
	We prove a nonexistence result for nonnegative solutions of the quasi-linear
	elliptic inequality
	\[
	-\Delta_p u\ge \sigma(x)u^q
	\]
	on infinite locally finite connected weighted graphs, where $1<p<\infty$ and
	$q>p-1$, $\sigma$ is a nonnegative Radon measure.  Under the non-$p$-parabolic setting, we show that every
	nonnegative solution is identically zero, provided the volume of intrinsic balls satisfy
	\[
	     \int_1^\infty
	\frac{r^{\frac{pq}{p-1}-1}}
	{\nu(B_\rho(o,r))^{\frac{q-p+1}{p-1}}}
	\dd r
	=\infty,
%
	\]
	This criterion recovers the known sharp pointwise critical volume-growth
	threshold and is strictly more flexible, since it allows irregular growth and
	does not require uniform upper bounds at every large radius.  The proof adapts
	the finite-network current method to the $p$-Laplace setting, combining a path
	decomposition with one-dimensional Hardy estimates, $p$-parallel-sum bounds
	across metric cuts, and the global $p$-Green function furnished by
	non-$p$-parabolicity.
\end{abstract}

\subjclass[2020]{Primary 35J92, 35R02; Secondary 31C20}
\keywords{weighted graph, $p$-Laplace operator, elliptic inequality, volume growth, Liouville theorem}

\maketitle

\section{Introduction}
Let $(V,E,\mu)$ be an infinite, connected, locally finite weighted graph, and
we aim to study the existence and nonexistence of positive solution to
\begin{equation}\label{eq:main-ineq}
	-\Delta_p u\ge \sigma(x)u^q
	\qquad\text{on }V,
\end{equation}
where
$
q>p-1>0$,and $ \sigma:V\to(0,\infty)$.

Throughout this
paper we assume that $(V,E,\mu)$ is not $p$-parabolic, otherwise the only nonnegative solution to (\ref{eq:main-ineq}) is identically zero; The equivalent
formulations of $p$-parabolicity used below are recalled in
Proposition~\ref{prop:parabolic_equiv}.

Liouville type theorems for elliptic inequalities of (\ref{eq:main-ineq}) on graphs have been studied
extensively in recent years.  For the semilinear inequality
$
-\Delta u\ge u^q ,
$
Gu, Huang and Sun \cite{GHS23} obtained existence and nonexistence criteria on
weighted graphs.  In particular, under the condition \textnormal{(P$_0$)}, they proved that if for some $o\in V$ and, 
\[
\mu(B(o,n))\le Cn^{\frac{2q}{q-1}}
(\log n)^{\frac1{q-1}},\quad\mbox{for
	$n\gg1$,}
\]
then the inequality admits no positive solution.  This volume growth condition here is
sharp at the logarithmic scale. Recently, the same authors obtained that if 
\begin{align}
\sum_{n=1}^{\infty}\frac{n^{2q-1}}{\mu(B(o,n))^{q-1}}=\infty
\end{align}
then every nonnegative solution of $-\Delta u\geq u^q$ is identically zero. We emphasize that there is no assumption on the  
condition \textnormal{(P$_0$)} any more.

The corresponding nonlinear problem
\[
-\Delta_p u\ge u^q
\]
was studied by Ge and Wang \cite{GW25}.  Their result extends the above
semilinear theorem to the quasilinear setting: for $q>p-1$, under the
analogous graph assumptions, the volume growth condition
\[
\mu(B(o,n))\le Cn^{\frac{pq}{q-p+1}}
(\log n)^{\frac{p-1}{q-p+1}},
\qquad\mbox{for $n\gg1$},
\]
implies the nonexistence of positive solutions. The sharp volume growth condition for quasi-linear differential inequality involving gradient terms in this direction is recently
established in \cite{ DLS26, HS23}.

On noncompact complete Rimennian manifolds, for the inequality 
$-\Delta_pu\geq u^q$, Sun proved a volume growth nonexistence criteria with sharp logarithmic expoents  in \cite{Sun16}.
 However, there is no sharp integral volume growth condition both on manifolds and weighted graph
obtained for (\ref{eq:main-ineq}) except for $p=2$, see \cite{GSV20,GHHS26}.

The aim of this paper is to prove a sharper volume-growth nonexistence theorem
for the $p$-Laplace inequality in the non-$p$-parabolic case.  Instead of a
pointwise upper bound on the ball volume, we use a integral form of intrinsic balls (see Section \ref{pre}).  More
precisely, our main result is the following.

\begin{theorem}\label{thm:main}
	Let $(V,E,\mu)$ be an infinite, connected, locally finite weighted graph.  Let
	$1<p<\infty$, $q>p-1$, and $\sigma:V\to(0,\infty)$.  Put
	$\nu=\sigma\mu$.  Assume that there exists a $p$-adapted edge length $\rho$
	satisfying \eqref{eq:adapted}, and assume that $(V,d_\rho)$ is complete.
	If, for some $o\in V$,
	\begin{equation}\label{eq:volume-condition}
		\int_1^\infty
		\frac{r^{\frac{pq}{p-1}-1}}
		{\nu(B_\rho(o,r))^{\frac{q-p+1}{p-1}}}
		\dd r
		=\infty,
	\end{equation}
	then every nonnegative solution of
	\begin{equation*}
		-\Delta_p u\ge \sigma(x)u^q
		\qquad\text{on }V
	\end{equation*}
	is identically zero.
\end{theorem}
If $\sigma\equiv1$,  \ref{eq:volume-condition} in  Theorem \ref{thm:main} can be reduced to the following series condition.
\begin{theorem}\label{thm:main2}
	Let $(V,E,\mu)$ be an infinite, connected, locally finite weighted graph.  Let
	$1<p<\infty$, $q>p-1$.
	If
	\begin{equation}\label{eq:volume-condition2}
		\sum_{n=1}^{\infty}
		\frac{n^{\frac{pq}{p-1}-1}}
		{\mu(B(o,n))^{\frac{q-p+1}{p-1}}}
		=
		\infty,
	\end{equation}
	then every nonnegative solution of
	\[
	-\Delta_p u\ge u^q
	\]
	is identically zero.
\end{theorem}

The condition \eqref{eq:volume-condition} is sharper than the corresponding
volume assumptions in \cite{GHS23} and \cite{GW25}; in particular, it implies
nonexistence under weaker and non-pointwise growth hypotheses.  The proof
follows the finite-network current method used in the semilinear case in
\cite{GHHS26}.  We use the path decomposition of \cite[Lemma~6.3]{GHHS26}; the
remaining estimates are replaced by their $p$-versions, namely a Hardy
inequality along paths and a $p$-parallel-sum estimate across metric cuts.

\section{Preliminaries}\label{pre}
Let \((V, E, \mu)\) be an infinite, connected, locally finite
weighted graph, where $V$ is the vertices set, and $E$ be the edge set.
 The
edge weights satisfy
\[
\mu_{xy}=\mu_{yx}>0 \quad \text{if } x\sim y,
\]
and the vertex measure is
\[
\mu(x):=\sum_{y\sim x}\mu_{xy}.
\]
For $1<p<\infty$, set
\[
\Phi_p(t):=|t|^{p-2}t.
\]
we can define the normalized graph $p$-Laplacian as
\[
\Delta_p u(x)
:=\frac1{\mu(x)}\sum_{y\sim x}\mu_{xy}\Phi_p(u(y)-u(x)).
\]
When $p=2$, we write $\Delta=\Delta_2$ for brevity.

Define the weighted vertex measure
\begin{align}\label{def:nu}
\nu(x):=\sigma(x)\mu(x),
\end{align}
and for $A\subset V$
\[
\nu(A):=\sum_{x\in A}\nu(x).
\]

For any two vertices $x$ and $y$, let $d(x,y)$
be the minimal number of edges among all possible paths connecting $x$ and $y$ on  graph $(V,\mu)$, then
$d(\cdot,\cdot)$ is a distance function on $V\times V$, and called the graph distance.
Fix some vertex $o\in V$, and for $r>0$, denote
$$B(o,r):=\{x\in V|\ d(o,x)\leq r\},$$
and
\begin{equation*}\label{Vol}
	V(o,r):=\mu(B(o,r)).
\end{equation*}

Beside the graph distance, we also use the intrinsic metric, which can be defined by the following adapted intrinsic length.
\begin{definition}[Adapted intrinsic length]
	A positive symmetric edge length $\rho:E\to(0,\infty)$ is called
	$p$-adapted to $\nu$ if
	\begin{equation}\label{eq:adapted}
		\sum_{y\sim x}\mu_{xy}\rho(x,y)^p\le \nu(x),
		\qquad \mbox{for any $x\in V$}.
	\end{equation}
\end{definition}

\begin{remark}
	Such $p$-adapted length $\rho(\cdot,\cdot)$ exists, for example,
	\[
	\rho(x,y):=\min\{\sigma(x),\sigma(y)\}^{1/p}
	\]
	is $p$-adapted, because $\rho(x,y)^p\le \sigma(x)$ for every neighbor $y$ of
	$x$.  The completeness/properness of $d_\rho$ is a separate issue.  In the
	statement below we assume that $d_\rho$ is complete, equivalently, in the
	usual discrete Hopf--Rinow setting for locally finite graphs, that all
	$d_\rho$-balls are finite.
\end{remark}

We define the induced path metric by the minimal weighted length between vertices $x\neq y$ by
\[
d_\rho(x,y)
:=\inf\left\{\sum_{i=1}^m \rho(x_{i-1},x_i):
x=x_0\sim x_1\sim\cdots\sim x_m=y\right\}.
\]

This metric is the so-called intrinsic path metric, and can be viewed as an analogue of the geodesic distance function
on a Riemannian manifold.

Now fix $o\in V$ and $r>0$,
let us write
\[
B_\rho(o,r):=\{x\in V:d_\rho(o,x)\le r\},
\]
and
\[\nu(B_{d_{\rho}}(o,r)):=\sum_{x\in B_{d_{\rho}}(o,r)}\nu(x).\]

For finitely supported function $\varphi$ on $V$, write
\[
       \E_p(f,\varphi)
       =
       \sum_{\{x,y\}\in E}
       \mu_{xy}\Phi_p(f(x)-f(y))(\varphi(x)-\varphi(y)),
\]
where each unordered edge is counted once. 

Let $\Omega\subset V$ be finite and let $o\in\Omega$.  The local $p$-Green
function $g_\Omega(\cdot)=g_\Omega(\cdot,o)$ is the unique solution of the Dirichlet problem
\begin{equation}\label{eq:finite-green-dirichlet}
\begin{cases}
       g_\Omega(x)=0, & x\in \Omega^c,\\
       -\Delta_p g_\Omega(x)=0, & x\in \Omega\setminus\{o\},\\
       -\Delta_p g_\Omega(o)=\mu(o)^{-1}.&
\end{cases}
\end{equation}
The above is also equivalent to that for every test function $\varphi$ supported in $\Omega$,
\begin{align}\label{eq:green-weak}
       \E_p(g_\Omega,\varphi)=\varphi(o).
\end{align}
 For $R\ge1$, put
\[
       g_R(x,y)=g_{B_R}(x,y),
\]
and write $g_R(x)=g_R(x,o)$.  Thus
\begin{equation}\label{eq:green-normalization}
       \E_p(g_R,\varphi)=\varphi(o)
\end{equation}
for every test function $\varphi$ supported in $B_R$.

For a finite nonempty set $A\subset V$, define
\[
       \operatorname{cap}_p(A)
       =
       \inf\left\{
       \sum_{\{x,y\}\in E}\mu_{xy}|v(x)-v(y)|^p:
       v\in C_c(V),\ v\ge1\text{ on }A
       \right\}.
\]
Noting that $g_R(x,y)$ is increasing in $R$, define the global $p$-Green
function by
\[
       g(x,y):=\lim_{R\to\infty}g_R(x,y).
\]

We say that $(V,E,\mu)$ is $p$-parabolic if it satisfies any of the equivalent
conditions in Proposition~\ref{prop:parabolic_equiv}.

\begin{proposition}\label{prop:parabolic_equiv}
The following statements are equivalent:
\begin{enumerate}[label=(\roman*)]
       \item For every finite nonempty subset $A\subset V$,
       $\operatorname{cap}_p(A)=0$.
       \item Every positive $p$-superharmonic function on $V$ is constant.
       \item For some $x,y\in V$, $g(x,y)=\infty$.
\end{enumerate}
\end{proposition}

\begin{proof}
The equivalence of these statements is a well-known result in nonlinear
potential theory on graphs; see  \cite{SC95, SC97} or \cite{AFS25} even for a more general
setting.
\end{proof}

\begin{lemma}\label{lem:1}
Let $u\ge0$ be $p$-superharmonic.  If $u(o)=0$ for some $o\in V$, then
$u\equiv0$ on $V$.  Hence every nontrivial nonnegative $p$-superharmonic
function is strictly positive.
\end{lemma}

\begin{proof}
Note at $o$,
\[
       0\le -\Delta_p u(o)
       =
       \frac1{\mu(o)}
       \sum_{y\sim o}\mu_{oy}\Phi_p(u(o)-u(y))
       =
       -\frac1{\mu(o)}
       \sum_{y\sim o}\mu_{oy}u(y)^{p-1}\le0.
\]
Thus $u(y)=0$ for all neighbors $y$ of $o$.  Repeating along paths and using
connectedness gives $u\equiv0$.
\end{proof}

We define the nonlinear weighted Green mass as
\begin{equation}\label{eq:LR-def}
	L_R^\nu
	:=\sum_{x\in\Omega_R}g_R(x)^q\nu(x).
\end{equation}
Before giving the a priori  estimate of $L_R^\nu$, let us introduce  the following elementary Picone-type edge inequality.
\begin{lemma}[Edge Picone inequality]\label{lem:picone}
	Let $a,b,s,t\ge0$, $1<p<\infty$, and $q>p-1$.  With
	$\eta=q-p+1$, one has
	\begin{equation}\label{eq:edge-picone}
		\Phi_p(a-b)(s^q-t^q)
		\le
		\frac{\sigma}{\eta}
		\Phi_p(as-bt)(s^\eta-t^\eta).
	\end{equation}
\end{lemma}

\begin{proof}
	If $s=t$, there is nothing to prove.  By switching the pairs $(a,s)$ and
	$(b,t)$ if necessary, assume $s>t$.  For every $\lambda\in[t,s]$,
	\[
	as-bt-\lambda(a-b)=a(s-\lambda)+b(\lambda-t)\ge0.
	\]
	Since $\Phi_p$ is increasing, we have
	\[
	\lambda^{p-1}\Phi_p(a-b)=\Phi_p(\lambda(a-b))
	\le \Phi_p(as-bt).
	\]
	Using $q-1=(p-1)+(\eta-1)$, we obtain
	\[
	\begin{aligned}
		\Phi_p(a-b)(s^q-t^q)
		&=q\int_t^s \lambda^{q-1}\Phi_p(a-b)\,\dd\lambda  \\
		&=q\int_t^s \lambda^{\eta-1}\lambda^{p-1}
		\Phi_p(a-b)\,\dd\lambda                                      \\
		&\le q\int_t^s \lambda^{\eta-1}\Phi_p(as-bt)\,\dd\lambda  \\
		&=\frac{q}{\eta}\Phi_p(as-bt)(s^\eta-t^\eta),
	\end{aligned}
	\]
	which completes the proof.
\end{proof}

\begin{lemma}[Uniform upper bound]\label{lem:upper}
	Assume that there exists a nontrivial nonnegative solution of
	\eqref{eq:main-ineq}.  Then, for every $R$ with $o\in\Omega_R$,
	\begin{equation}\label{eq:upper-estimate}
		L_R^\nu
		\le
		\frac q\eta
		\left(\frac{g_R(o)}{u(o)}\right)^\eta,
		\qquad
		\eta=q-p+1.
	\end{equation}
	Consequently, $L_R^\nu$ is bounded uniformly in $R$.
\end{lemma}

\begin{proof}
Let $u$ be a nontrivial nonnegative solution of \eqref{eq:main-ineq}, and from Lemma \ref{lem:1},	we obtain that $u>0$ on $V$.
	
	Set
	\[
	w_R(x):=\frac{g_R(x)}{u(x)}.
	\]
	Testing \eqref{eq:main-ineq} against $w_R^q$, which is supported in
	$\Omega_R$, yields
	\[
	L_R^\nu
	=\sum_{x\in\Omega_R}\sigma(x)u(x)^q w_R(x)^q\mu(x)
	\le \E_p(u,w_R^q).
	\]
	Applying Lemma~\ref{lem:picone} on each edge with
	\[
	a=u(x),\quad b=u(y),\quad s=w_R(x),\quad t=w_R(y),
	\]
	and using $aw_R(x)=g_R(x)$, $bw_R(y)=g_R(y)$, we get
	\[
	\E_p(u,w_R^q)
	\le \frac q\eta\E_p(g_R,w_R^\eta).
	\]
	By the Green normalization \eqref{eq:green-weak},
	\[
	\E_p(g_R,w_R^\eta)=w_R(o)^\eta
	=\left(\frac{g_R(o)}{u(o)}\right)^\eta.
	\]
	This proves \eqref{eq:upper-estimate}.
	
	Finally,
	\[
	L_R^\nu\ge g_R(o)^q\nu(o).
	\]
	Combining this with \eqref{eq:upper-estimate} gives
	\[
	\nu(o)g_R(o)^q
	\le
	\frac q\eta \frac{g_R(o)^\eta}{u(o)^\eta}.
	\]
	Since $q-\eta=p-1>0$ and $\nu(o)>0$, the numbers $g_R(o)$ are uniformly
	bounded.  Returning to \eqref{eq:upper-estimate}, we obtain a uniform bound
	for $L_R^\nu$.
\end{proof}

\section{Cable interpolation and current lower bound}\label{cable}

In this section, we  prove a lower bound for $L_R^\nu$ depending only on the intrinsic volume
of $\nu$-balls.

Fix $R>2$ and let $g_R$ be the Dirichlet $p$-Green function in
$\Omega_R=B_\rho(o,R)$.  Orient every edge with positive voltage drop from the
larger value of $g_R$ to the smaller one.  For a retained oriented edge
$e=(e^-,e^+)$, define
\[
\delta_e:=g_R(e^-)-g_R(e^+)>0,
\qquad
\theta_e:=\mu_e\delta_e^{p-1},
\]
where $\mu_e=\mu_{xy}$ for $e=\{x,y\}$.  The finite Green equation implies that
$\theta$ is an acyclic unit flow from $o$ to $V\setminus\Omega_R$.  Hence it
admits a path decomposition: there exists a probability measure on finite
voltage-decreasing paths $\gamma$ from $o$ to $V\setminus\Omega_R$ such that
\begin{equation}\label{eq:path-marginal}
	\mathbb P(\gamma\text{ uses }e)=\theta_e
\end{equation}
for every retained directed edge $e$, see Appendix \ref{appendix}.

For an edge $e$, write
\[
\rho_e:=\rho(e),
\qquad
p_e:=\mu_e\rho_e^{p-1}.
\]
We view $e$ as an interval of length $\rho_e$ and put line density $p_e\,ds$
on it.  Thus the whole cable edge has measure
\[
p_e\rho_e=\mu_e\rho_e^p.
\]
Let $\widehat d_\rho$ denote the corresponding cable length metric, and let
$\widehat\nu$ be the cable line measure.  Put
\begin{align}\label{def:M_o(r)}
M_o(r):=\nu(B_\rho(o,r)),
\qquad
\widehat M_o(r):=\widehat\nu(B_{\widehat d_\rho}(o,r)).
\end{align}
It follows that for a.e. $r>0$,
\begin{equation}\label{def-a}
\alpha(r):=\widehat M_o^{\prime}(r).
 \end{equation}
 Equivalently, $\alpha(r)$
is the sum of the cable densities $p_e$ over the cable points lying on the
level sphere $\partial B_{\widehat d_\rho}(o,r)$.

\begin{lemma}[Cable volume domination]\label{lem:cable-volume}
	For every $r>0$, there holds
	\[
	\widehat M_o(r)\le M_o(r).
	\]
\end{lemma}

\begin{proof}
	If a cable edge intersects $B_{\widehat d_\rho}(o,r)$, then a shortest path
	from $o$ to an intersection point enters that edge through an endpoint $x$
	with $d_\rho(o,x)\le r$.  Hence the cable ball is contained in the union of
	edge intervals incident to vertices in $B_\rho(o,r)$.  Therefore, by the
	$p$-adaptedness condition \eqref{eq:adapted},
	\[
	\begin{aligned}
		\widehat M_o(r)
		&\le
		\sum_{x\in B_\rho(o,r)}\sum_{y\sim x}\mu_{xy}\rho(x,y)^p  \\
		&\le \sum_{x\in B_\rho(o,r)}\nu(x)
		=M_o(r),
	\end{aligned}
	\]
	which completes the proof.
\end{proof}

\begin{lemma}[One-dimensional weighted Hardy inequality]\label{lem:hardy-cont}
	Let $\kappa>0$, $\eta>0$, and $q=\kappa+\eta$.  Let $v:[0,L]\to[0,\infty)$
	be continuous, strictly decreasing, and piecewise affine, with $v(L)=0$ and
	$-v'>0$ on each affine piece.  Then
	\begin{equation}\label{eq:hardy-cont}
		\int_0^L\frac{v(s)^q}{(-v'(s))^\kappa}\dd s
		\ge
		c_{\kappa,\eta}
		\int_0^L s^\kappa v(s)^\eta\dd s,
	\end{equation}
	where one may take
	\[
	c_{\kappa,\eta}
	=\frac{1}{2^{\kappa+1}\log2}
	\left(\frac{\eta}{\kappa}\right)^\kappa.
	\]
\end{lemma}

\begin{proof}
	Set
	\[
	\alpha:=\frac q\kappa>1,
	\qquad
	a(s):=\frac{-v'(s)}{v(s)^\alpha},
	\qquad
	A(t):=\int_0^t a(s)\dd s.
	\]
	Then
	\[
	\frac{v(s)^q}{(-v'(s))^\kappa}=a(s)^{-\kappa}.
	\]
	For every $t\in(0,L)$, applying Jensen's inequality on $[t/2,t]$ gives
	\[
	\int_{t/2}^t a(s)^{-\kappa}\dd s
	\ge
	\frac{(t/2)^{\kappa+1}}{\left(\int_{t/2}^t a(s)\dd s\right)^\kappa}
	\ge
	\frac{(t/2)^{\kappa+1}}{A(t)^\kappa}.
	\]
	Thus
	\[
	\left(\frac{t}{A(t)}\right)^\kappa
	\le
	\frac{2^{\kappa+1}}{t}
	\int_{t/2}^t a(s)^{-\kappa}\dd s.
	\]
	Integrating this inequality in $t$ and changing the order of integration,
	\[
	\int_0^L\left(\frac{t}{A(t)}\right)^\kappa\dd t
	\le
	2^{\kappa+1}\log2\int_0^L a(s)^{-\kappa}\dd s.
	\]
	On the other hand,
	\[
	A(t)=\int_0^t\frac{-v'(s)}{v(s)^\alpha}\dd s
	\le \frac{v(t)^{1-\alpha}}{\alpha-1}.
	\]
	Consequently
	\[
	A(t)^{-\kappa}
	\ge
	(\alpha-1)^\kappa v(t)^{\kappa(\alpha-1)}
	=
	\left(\frac{\eta}{\kappa}\right)^\kappa v(t)^\eta.
	\]
	Combining the preceding estimates gives \eqref{eq:hardy-cont}.
\end{proof}

\begin{corollary}[Hardy estimate along one path]\label{cor:path-hardy}
	Let
	\[
	0=s_0<s_1<\cdots<s_m=L,
	\qquad
	V_0>V_1>\cdots>V_m=0.
	\]
	Put
	\[
	\ell_i:=s_{i+1}-s_i,
	\qquad
	\delta_i:=V_i-V_{i+1},
	\qquad
	\kappa:=p-1,
	\qquad
	\eta:=q-p+1.
	\]
	Let $v$ be the piecewise affine interpolation of the values $V_i$.  Then
	\begin{equation}\label{eq:path-hardy}
		\sum_{i=0}^{m-1}\frac{\ell_i^p V_i^q}{\delta_i^{p-1}}
		\ge
		c_{p,q}\int_0^L s^{p-1}v(s)^\eta\dd s.
	\end{equation}
\end{corollary}

\begin{proof}
	On $(s_i,s_{i+1})$ one has $-v'=\delta_i/\ell_i$ and $v\le V_i$.  Therefore
	\[
	\int_0^L\frac{v(s)^q}{(-v'(s))^{p-1}}\dd s
	\le
	\sum_{i=0}^{m-1}\frac{\ell_i^p V_i^q}{\delta_i^{p-1}}.
	\]
	Apply Lemma~\ref{lem:hardy-cont} with $\kappa=p-1$, we can obtain (\ref{eq:path-hardy}).
\end{proof}

\begin{proposition}\label{prop:current-lower}
	For every $R>2$,
	\begin{equation}\label{eq:cut-lower}
		L_R^\nu
		\ge
		c_{p,q}
		\int_0^R t^{p-1}
		\left(
		\int_t^R \alpha(s)^{-1/(p-1)}\dd s
		\right)^{q-p+1}
		\dd t.
	\end{equation}
\end{proposition}

\begin{proof}
	Let us take a sampled current path
	\[
	\gamma=(x_0=o,e_0,x_1,\ldots,e_{m-1},x_m),
	\]
	where $x_m\notin\Omega_R$.  Write
	\[
	V_i:=g_R(x_i),
	\qquad
	\delta_i:=V_i-V_{i+1},
	\qquad
	\ell_i:=\rho(e_i).
	\]
	Using the path marginal identity \eqref{eq:path-marginal},
	\[
	\begin{aligned}
		\mathbb E_\gamma
		\sum_{i=0}^{m-1}\frac{\ell_i^pV_i^q}{\delta_i^{p-1}}
		&=
		\sum_e \theta_e
		\frac{\rho_e^p g_R(e^-)^q}{\delta_e^{p-1}}  \\
		&=
		\sum_e \mu_e\rho_e^p g_R(e^-)^q  \\
		&\le
		\sum_{x\in\Omega_R}g_R(x)^q
		\sum_{y\sim x}\mu_{xy}\rho(x,y)^p \\
		&\le
		\sum_{x\in\Omega_R}g_R(x)^q\nu(x)=L_R^\nu.
	\end{aligned}
	\]
	Parametrize the path by intrinsic arclength $s\in[0,L_\gamma]$ and let
	$V_\gamma(s)$ be the piecewise affine interpolation of the Green voltage along
	the path.  Corollary~\ref{cor:path-hardy} gives
	\begin{equation}\label{eq:LR-path}
		L_R^\nu
		\ge
		c_{p,q}\,
		\mathbb E_\gamma\int_0^{L_\gamma}s^{p-1}
		V_\gamma(s)^\eta\dd s,
		\qquad
		\eta=q-p+1.
	\end{equation}
	
	For $0<t<R$, let $T_t(\gamma)$ be the first time at which the path reaches the
	cable sphere $\partial B_{\widehat d_\rho}(o,t)$, and put
	\[
	W_t(\gamma):=V_\gamma(T_t(\gamma)).
	\]
	The path starts at $o$ and exits $B_\rho(o,R)$, hence it crosses every level
	$0<t<R$.  Since the distance traveled along the path is at least the increase
	of the radial distance, a first-hitting argument gives
	\begin{equation}\label{eq:first-hit-lower}
		\int_0^{L_\gamma}s^{p-1}V_\gamma(s)^\eta\dd s
		\ge
		\int_0^R t^{p-1}W_t(\gamma)^\eta\dd t.
	\end{equation}
	
	It remains to estimate $\mathbb E_\gamma W_t^\eta$ from below.  Fix a regular
	level $t\in(0,R)$.  For a.e. $s\in[t,R]$, let $z_s(\gamma)$ be the first point
	where $\gamma$ meets the sphere of radius $s$, and let
	\[
	\lambda_s(\gamma):=-V_\gamma'(z_s(\gamma))>0,
	\qquad
	Y_s(\gamma):=\lambda_s(\gamma)^{-(p-1)}.
	\]
	On a cable edge $e$, the current-slope identity is
	\[
	p_e\lambda_s^{p-1}
	=\mu_e\rho_e^{p-1}\left(\frac{\delta_e}{\rho_e}\right)^{p-1}
	=\mu_e\delta_e^{p-1}=\theta_e.
	\]
	Therefore, if the first crossing of level $s$ lies on $e$, then
	\[
	Y_s=\frac{p_e}{\theta_e}.
	\]
	The event that this point is selected implies that the path uses $e$, whose
	probability is $\theta_e$.  Summing over all cable points on the level sphere,
	we get
	\begin{equation}\label{eq:EY-alpha}
		\mathbb E_\gamma Y_s\le \alpha(s)
		\qquad\text{for a.e. }s.
	\end{equation}
	Moreover,
	\[
	W_t(\gamma)
	\ge
	\int_t^R\lambda_s(\gamma)\dd s
	=
	\int_t^R Y_s(\gamma)^{-1/(p-1)}\dd s.
	\]
	Using the finite-dimensional $p$-parallel-sum inequality and Jensen's
	inequality, \eqref{eq:EY-alpha} implies
	\begin{equation}\label{eq:level-lower}
		\mathbb E_\gamma W_t^\eta
		\ge
		\left(
		\int_t^R \alpha(s)^{-1/(p-1)}\dd s
		\right)^\eta.
	\end{equation}
	Indeed, one first proves the estimate for a partition of $[t,R]$ on which
	$\alpha$ is constant, using the concavity of
	\[
	(y_1,\ldots,y_N)\mapsto
	\left(\sum_i a_i y_i^{-1/(p-1)}\right)^{-(p-1)},
	\]
	and then passes to the step-function representation of $\alpha$.
	
	Combining \eqref{eq:LR-path}, \eqref{eq:first-hit-lower}, and
	\eqref{eq:level-lower} proves \eqref{eq:cut-lower}.
\end{proof}

\section{Proof of Theorem~\ref{thm:main} and \ref{thm:main2}.}
\begin{proposition}[Intrinsic volume lower bound]\label{prop:volume-lower}
	There exists $C_{p,q}>0$ such that, for every $R>4$,
	\begin{equation}\label{eq:volume-lower}
		L_R^\nu
		\ge
		C_{p,q}
		\int_2^R
		\frac{s^{\frac{pq}{p-1}-1}}
		{M_o(s)^{\frac{q-p+1}{p-1}}}
		\dd s,
	\end{equation}
	where $	M_o(s)=\nu(B_\rho(o,s))$ is defined in (\ref{def:M_o(r)}).
\end{proposition}

\begin{proof}
	From Proposition~\ref{prop:current-lower}, restricting the outer integral to
	$1\le t\le R/2$ and the inner integral to $[t,2t]$, we obtain
	\[
	L_R^\nu
	\ge
	c_{p,q}\int_1^{R/2}t^{p-1}
	\left(\int_t^{2t}\alpha(s)^{-1/(p-1)}\dd s\right)^\eta\dd t,
	\qquad
	\eta=q-p+1.
	\]
	By using Jensen's inequality, we have
	\[
	\int_t^{2t}\alpha(s)^{-1/(p-1)}\dd s
	\ge
	\frac{t^{p/(p-1)}}
	{\left(\int_t^{2t}\alpha(s)\dd s\right)^{1/(p-1)}}.
	\]
	Since
	\[
	\int_t^{2t}\alpha(s)\dd s
	=\widehat M_o(2t)-\widehat M_o(t)
	\le \widehat M_o(2t)
	\le M_o(2t)
	\]
	by Lemma~\ref{lem:cable-volume}, we get
	\[
	L_R^\nu
	\ge
	c_{p,q}\int_1^{R/2}
	\frac{t^{p-1+\frac{p\eta}{p-1}}}
	{M_o(2t)^{\eta/(p-1)}}\dd t.
	\]
	The exponent satisfies
	\[
	p-1+\frac{p\eta}{p-1}
	=\frac{pq}{p-1}-1.
	\]
	Changing variables $s=2t$ and odifying the
	constant gives \eqref{eq:volume-lower}.
\end{proof}

Now we are ready to give the proof of Theorem~\ref{thm:main}.
\begin{proof}[Proof of Theorem~\ref{thm:main}]
	Assume, toward a contradiction, that $u\ge0$ is a nontrivial solution of
	\eqref{eq:main-ineq}.  By Lemma~\ref{lem:upper}, the quantities $L_R^\nu$ are
	bounded uniformly in $R$.
	
	On the other hand, Proposition~\ref{prop:volume-lower} gives
	\[
	L_R^\nu
	\ge
	C_{p,q}
	\int_2^R
	\frac{s^{\frac{pq}{p-1}-1}}
	{\nu(B_\rho(o,s))^{\frac{q-p+1}{p-1}}}
	\dd s.
	\]
	The assumed divergence condition \eqref{eq:volume-condition} implies that the
	right-hand side tends to $+\infty$ as $R\to\infty$.  This contradicts the
	uniform upper bound.  Therefore no nontrivial nonnegative solution exists, and
	hence every nonnegative solution is identically zero.
\end{proof}


The proof of Theorem \ref{thm:main2} depends on a chosen intrinsic metric.
\begin{proof}
	Choose the constant edge length
	\[
	\rho(x,y):=1.
	\]
	Then
	\[
	\sum_{y\sim x}\mu_{xy}\rho(x,y)^p
	=\mu(x)=\nu(x),
	\]
	so $\rho$ is $p$-adapted.  Moreover
	\[
	d_\rho=d,
	\qquad
	\nu(B_\rho(o,r))
	=\mu\left(B\left(o,r\right)\right).
	\]
It follows that $(V, d)$ is complete.	Thus the integral condition \eqref{eq:volume-condition} is equivalent to the standard integral-series comparison
	\eqref{eq:volume-condition2}.  The conclusion follows from
	Theorem~\ref{thm:main}.
\end{proof}

\section{Appendix: Path decomposition }\label{appendix}
Let $G=(V,\mu)$ be a finite connected weighted graph, where $\mu_{xy}=\mu_{yx}>0$ is the conductance of the unoriented edge $\{x,y\}$.  We distinguish two vertices, a source $o$ and a sink $\partial$.  In the application, $G$ is obtained from an infinite graph by collapsing $B(o,R)^c$ to the single vertex $\partial$.

We use the unnormalized $p$-Laplacian
\begin{equation*}
	-\Delta_ph(x)=\mathcal L h(x)
	:=\sum_{y\sim x} \mu_{xy}\Phi_p\bigl(h(x)-h(y)\bigr).
\end{equation*}
which is the natural operator for currents.  It differs from the normalized graph Laplacian by multiplication by the vertex measure.

Let $h$ be the voltage for one unit of current sent from $o$ to $\partial$.  Thus
\begin{equation}\label{eq:unit-current-voltage}
	h(\partial)=0, \qquad
	\mathcal L h(x)=0 \quad (x\neq o,\partial), \qquad
	\mathcal L h(o)=1.
\end{equation}
Equivalently, the total current leaving $o$ is one.  In Section \ref{cable}, $h$ is the normalized Green kernel \(g_{B_R}(o,\cdot)\).

For each unoriented edge $\{x,y\}$ with $h(x)>h(y)$, orient the edge from $x$ to $y$ and define
\begin{equation*}
	\theta_{xy}:=\mu_{xy}\bigl(h(x)-h(y)\bigr)^{p-1}>0.
\end{equation*}
Edges with $h(x)=h(y)$ carry no current and are discarded.  Let $\cE$ denote the resulting set of directed positive-current edges.

For a directed edge $e=(x,y)\in\cE$, we write
\begin{equation*}
	\tail(e)=x, \qquad \head(e)=y, \qquad
	\mu_e=\mu_{xy}, \qquad \theta_e=\theta_{xy}.
\end{equation*}

We first isolate the flow properties of $\theta$.

\begin{lemma}[Electrical current is an acyclic unit flow]\label{lem:current-is-flow}
	The directed weighted graph $(V,\cE,\theta)$ has the following properties.
	\begin{enumerate}
		\item It is acyclic.
		\item No positive-current edge enters $o$, and no positive-current edge leaves $\partial$.
		\item If
		\begin{equation*}
			\theta^+(x):=\sum_{(x,y)\in\cE}\theta_{xy},
			\qquad
			\theta^-(x):=\sum_{(y,x)\in\cE}\theta_{yx},
		\end{equation*}
		then
		\begin{equation*}
			\theta^+(o)=1, \qquad \theta^-(o)=0,
		\end{equation*}
		\begin{equation*}
			\theta^+(x)=\theta^-(x) \qquad (x\neq o,\partial),
		\end{equation*}
		and
		\begin{equation*}
			\theta^+(\partial)=0, \qquad \theta^-(\partial)=1.
		\end{equation*}
	\end{enumerate}
\end{lemma}

\begin{proof}
	For acyclicity, suppose that there were a directed cycle
	\begin{equation*}
		x_0\to x_1\to \cdots \to x_m=x_0.
	\end{equation*}
	Every directed edge goes strictly downhill in voltage, so
	\begin{equation*}
		h(x_0)>h(x_1)>\cdots>h(x_m)=h(x_0),
	\end{equation*}
	which is impossible.  Hence the directed graph is acyclic.
	
	By the maximum principle for the finite Dirichlet problem, $h$ attains its maximum at the source $o$ and its minimum at the sink $\partial$.  Since $h(\partial)=0$ and one unit of current is injected at $o$, the voltage is nonconstant.  Thus positive-current edges cannot enter $o$, and positive-current edges cannot leave $\partial$.
	
	It remains to verify the balance identities.  For any vertex $x$, split the conductance sum into neighbors below and above $x$ in voltage.  Edges with equal voltage contribute zero.  Thus
	\begin{align*}
		\mathcal L h(x)
		&=\sum_{y\sim x} \mu_{xy}\Phi_p\bigl(h(x)-h(y)\bigr) \\
		&=\sum_{(x,y)\in\cE}\theta_{xy}
		-\sum_{(y,x)\in\cE}\theta_{yx} \\
		&=\theta^+(x)-\theta^-(x).
	\end{align*}
	Using \eqref{eq:unit-current-voltage}, we obtain
	\begin{equation*}
		\theta^+(o)-\theta^-(o)=1, \qquad
		\theta^+(x)-\theta^-(x)=0 \quad (x\neq o,\partial).
	\end{equation*}
	Since no positive-current edge enters $o$, $\theta^-(o)=0$, and hence $\theta^+(o)=1$.  Summing the identity $\theta^+(x)-\theta^-(x)$ over all vertices gives zero, because each directed edge is counted once with a plus sign and once with a minus sign.  Therefore
	\begin{equation*}
		\mathcal L h(\partial)=-1.
	\end{equation*}
	Since no positive-current edge leaves $\partial$, this is exactly $-\theta^-(\partial)=-1$, so $\theta^-(\partial)=1$.  The interior balance follows directly from $\mathcal Lh(x)=0$.
\end{proof}

\begin{lemma}[Path decomposition of an acyclic unit flow]\label{lem:path-decomposition-abstract}
	Let $(V,\cE)$ be a finite directed acyclic graph with source $o$ and sink $\partial$.  Let $\theta:\cE\to(0,\infty)$ be a unit flow, meaning that
	\begin{equation*}
		\theta^+(o)=1, \qquad \theta^-(o)=0,
	\end{equation*}
	\begin{equation}\label{eq:abstract-flow-interior}
		\theta^+(x)=\theta^-(x) \qquad (x\neq o,\partial),
	\end{equation}
	\begin{equation*}
		\theta^+(\partial)=0, \qquad \theta^-(\partial)=1,
	\end{equation*}
	where
	\begin{equation*}
		\theta^+(x)=\sum_{(x,y)\in\cE}\theta_{xy},
		\qquad
		\theta^-(x)=\sum_{(y,x)\in\cE}\theta_{yx}.
	\end{equation*}
	Then there is a probability measure $\mathbb P$ on the finite set $\cP$ of directed paths from $o$ to $\partial$ such that, for every directed edge $e\in\cE$,
	\begin{equation}\label{eq:edge-probability-current}
		\mathbb P(\gamma\in\cP: e\in\gamma)=\theta_e.
	\end{equation}
	Equivalently, there are nonnegative numbers $(\lambda_\gamma)_{\gamma\in\cP}$ such that
	\begin{equation}\label{eq:path-weights-sum}
		\sum_{\gamma\in\cP}\lambda_\gamma=1
	\end{equation}
	and
	\begin{equation}\label{eq:path-weights-edge}
		\theta_e=\sum_{\gamma\in\cP: e\in\gamma}\lambda_\gamma
		\qquad (e\in\cE).
	\end{equation}
\end{lemma}

\begin{proof}
	We give a probabilistic construction.  Because the graph is acyclic, there is a topological ordering of the vertices.  In particular, a directed walk cannot visit the same vertex twice.
	
	For every vertex $x\neq\partial$ with $\theta^+(x)>0$, define transition probabilities by
	\begin{equation}\label{eq:transition-probability}
		p(x,y):=\frac{\theta_{xy}}{\theta^+(x)}
		\qquad ((x,y)\in\cE).
	\end{equation}
	At vertices with $\theta^+(x)=0$ no transition is needed.  Starting from $x_0=o$, move according to these transition probabilities until the sink $\partial$ is reached.
	
	We first check that this procedure is well-defined and stops at $\partial$.  At the source, $\theta^+(o)=1$, so the first step is defined.  If the process reaches a vertex $x\neq\partial$, then either $x=o$ or $x$ has positive incoming flow from the edge just used.  In the latter case, \eqref{eq:abstract-flow-interior} gives
	\begin{equation*}
		\theta^+(x)=\theta^-(x)>0,
	\end{equation*}
	so the next transition is defined.  Since the directed graph is finite and acyclic, the process cannot continue forever.  Therefore it must stop at a vertex with no outgoing flow.  The only such reachable vertex is $\partial$, because every reachable non-sink vertex has positive outgoing flow.  Thus the procedure produces a directed path
	\begin{equation*}
		\gamma=(x_0=o,e_0,x_1,e_1,\ldots,e_{m-1},x_m=\partial).
	\end{equation*}
	This defines a probability measure on $\cP$.
	
	It remains to compute the probability of using a given edge.  Let
	\begin{equation*}
		r(x):=\mathbb P(\gamma \text{ visits } x).
	\end{equation*}
	We claim that
	\begin{equation}\label{eq:visit-probability-flow}
		r(x)=\theta^+(x) \qquad (x\neq\partial),
	\end{equation}
	and $r(\partial)=1$.  For the source this is true because $r(o)=1=\theta^+(o)$.
	
	Now take a vertex $x\neq o,\partial$.  Since the graph is acyclic, the path can visit $x$ only by first using exactly one of the edges entering $x$.  Hence
	\begin{align*}
		r(x)
		&=\sum_{(z,x)\in\cE}
		\mathbb P(\gamma \text{ visits } z)\,p(z,x).
	\end{align*}
	Using induction along a topological ordering and the definition \eqref{eq:transition-probability}, we obtain
	\begin{align*}
		r(x)
		&=\sum_{(z,x)\in\cE}
		\theta^+(z)\frac{\theta_{zx}}{\theta^+(z)} \\
		&=\sum_{(z,x)\in\cE}\theta_{zx}
		=\theta^-(x)
		=\theta^+(x),
	\end{align*}
	where the final equality is the interior flow balance.  This proves \eqref{eq:visit-probability-flow}.
	
	The same computation at the sink gives
	\begin{align*}
		r(\partial)
		&=\sum_{(z,\partial)\in\cE}
		r(z)\,p(z,\partial) \\
		&=\sum_{(z,\partial)\in\cE}
		\theta^+(z)\frac{\theta_{z\partial}}{\theta^+(z)} \\
		&=\sum_{(z,\partial)\in\cE}\theta_{z\partial}
		=\theta^-(\partial)=1.
	\end{align*}
	Thus the path reaches $\partial$ with probability one.
	
	Finally fix an edge $e=(x,y)\in\cE$.  The event that $\gamma$ uses $e$ is the event that $\gamma$ visits $x$ and then chooses $y$ as its next vertex.  Therefore
	\begin{align*}
		\mathbb P(e\in\gamma)
		&=r(x)\,p(x,y) \\
		&=\theta^+(x)\frac{\theta_{xy}}{\theta^+(x)}
		=\theta_{xy}=\theta_e.
	\end{align*}
	This proves \eqref{eq:edge-probability-current}.  Hence, the proof is complete.
\end{proof}


\end{document}